\documentclass[10pt,notitlepage,twoside,a4paper]{amsart}

\usepackage{amsmath,amssymb,enumerate}

\usepackage{epsfig,fancyhdr,color}

\usepackage{amssymb}
\usepackage{amsmath,amsthm}
\usepackage{latexsym}
\usepackage{amscd}
\usepackage{psfrag}
\usepackage{graphicx}
\usepackage[latin1]{inputenc}
\usepackage[all]{xy}
\usepackage[mathcal]{eucal}

\definecolor{NoteColor}{rgb}{1,0,0}

\renewcommand{\textsc}{\textcolor{red}}

\newtheorem{theorem}{\rm\bf Theorem}[section]
\newtheorem{proposition}[theorem]{\rm\bf Proposition}
\newtheorem{lemma}[theorem]{\rm\bf Lemma}
\newtheorem{corollary}[theorem]{\rm\bf Corollary}
\newtheorem*{theorem 1}{\rm\bf Proposition 1}
\newtheorem*{theorem 2}{\rm\bf Proposition 2}

\theoremstyle{definition}
\newtheorem{definition}[theorem]{\rm\bf Definition}

\theoremstyle{remark}
\newtheorem{remark}[theorem]{\rm\bf Remark}

\def\interieur#1{\mathord{\mathop{\kern 0pt #1}\limits^\circ}}

\title[Optimal Lipschitz maps on one-holed tori]{Optimal Lipschitz maps on one-holed tori and the Thurston metric theory of Teichm\"uller space}

\author{Yi Huang}

 \address{Yi Huang, Office 261 Jinchunyuan West Building, Yau Mathematical Sciences Center,
Tsinghua University, Haidian District
Beijing 100084, China}
 \email{yihuangmath@tsinghua.edu.cn}

\author{Athanase Papadopoulos}

\address{Athanase Papadopoulos,  Institut de Recherche Math\'ematique Avanc\'ee, Universit{\'e} de Strasbourg and CNRS,
7 rue Ren\'e Descartes,
 67084 Strasbourg Cedex, France}
\email{athanase.papadopoulos@math.unistra.fr}

\date{\today}

\begin{document}

\maketitle
\begin{abstract}

We study Thurston's Lipschitz and curve metrics, as well as the arc metric on the Teichm\"uller space of one-hold tori equipped with complete hyperbolic metrics with boundary holonomy of fixed length. We construct natural Lipschitz maps between two surfaces equipped with such hyperbolic metrics that generalize Thurston's stretch maps and prove the following: 
(1) On the Teichm\"uller space of the torus with one boundary component, the Lipschitz and the curve metrics coincide and define a geodesic metric on this space. 
(2) On the same space, the arc and the curve metrics coincide when the length of the boundary component is $\leq 4\operatorname{arcsinh}(1)$, but differ when the boundary length is large. 
We further apply our stretch map generalization to construct novel Thurston geodesics on the Teichm\"uller spaces of closed hyperbolic surfaces, and use these geodesics to show that the sum-symmetrization of the Thurston metric fails to exhibit Gromov hyperbolicity.

The final version of this paper will appear in Geometriae Dedicata.
\medskip

\noindent  
\emph{Keywords.---} Teichm\"uller space, hyperbolic surface, Lipschitz metric, curve metric, Thurston metric, arc metric, stretch map, stretch path,  partial stretch path, geodesic, Gromov hyperbolicity.
  
\medskip

\noindent    
\emph{AMS classification.---}  32G15, 30F60, 30F10, 53C23, 53C70.
\end{abstract}
     
\medskip
     
  
\section{Introduction}\label{s:intro}
This paper is concerned with Thurston's theory of Lipschitz maps and minimal stretch maps between hyperbolic surfaces and their applications to Teichm\"uller spaces.

\subsection{Background and context}

In his 1986 preprint \emph{Minimal stretch maps between hyperbolic surfaces} \cite{Thurston1986}, Thurston introduced on Teichm\"uller space an asymmetric metric (that is, a distance function which satisfies all the axioms of a metric except the symmetry axiom), which we henceforth refer to as Thurston's metric.  Thurston's motivation was to develop a metric theory of Teichm\"uller space which is completely based on the rigid geometry of hyperbolic surfaces, in contrast to the Teichm\"uller metric, whose development is much more analytical in nature. To supplant the role of quasiconformal mappings in the analytical classical approach to Teichm\"uller theory, Thurston described a class of maps between hyperbolic surfaces which he dubbed \emph{stretch maps}, based on a certain family of geometrically defined Lipschitz maps between hyperbolic ideal triangles. Using stretch maps, he constructed a class of geodesics for Thurston's metric, which he called stretch lines. In particular, Thurston showed that two arbitrary points in Teichm\"uller space are joined by a geodesic realized as a concatenation of stretch lines, thereby showing that Thurston's metric is geodesic. He further showed that Thurston's metric is Finsler and initiated a systematic study of the structure of tangent spaces of Teichm\"uller space equipped with the norm induced by this Finsler structure.

There are strong analogies between Thurston's results in \cite{Thurston1986} and well-known results regarding the Teichm\"uller metric, and much of this owes to similarities between Thurston's stretch maps and Teichm\"uller mappings. Thurston's stretch maps constitute special examples of optimal Lipschitz maps between hyperbolic surfaces $(S,h_1)$ and $(S,h_2)$, stretching the leaves of a geodesic lamination and contracting the leaves of a partially defined measured foliation which orthogonally intersects the leaves of the lamination, whereas Teichm\"uller mappings yield optimal Lipschitz maps between collections $\{(S,\hat{h}_1)\}$ and $\{(S,\hat{h}_2)\}$ of special singular Euclidean metrics, stretching the leaves of a measured foliation, which are geodesics for the underlying Euclidean structure, and contracting the leaves of a transverse measured foliation which intersects it orthogonally, which are also geodesics. The (a priori) flexibility in the domain and codomain for the singular Euclidean metric Lipschitz optimization problem contrasts with the uniqueness of its solution map: there is a unique Teichm\"uller mapping which optimizes the Lipschitz constant between certain (unique) domain and codomain metrics among the $\{(S,\hat{h}_i)\}$. On the other hand, for the hyperbolic metric Lipschitz optimization problem, the domain and codomain metrics are uniquely specified, but there is flexibility in the optimal map.\medskip

Thurston's theory allows for hyperbolic surfaces (of finite type) with cusps, but not for surfaces with non-empty geodesic boundary.  Recall that in Thurston's setting, there are two equivalent ways of describing Thurston's metric:
\begin{itemize}
\item
one approach uses the best Lipschitz constant among homeomorphisms between marked hyperbolic surfaces as a way of defining a distance between those surfaces. We call the asymmetric metric defined in this manner the  \emph{Lipschitz metric} (see Definition~\ref{def:Lip});
\item
the other approach uses a ratio comparison between lengths of simple closed geodesics, and we refer to the asymmetric metric defined in this manner the \emph{curve metric} (see Definition \ref{def:curve}).
\end{itemize}
The equality between the Lipschitz metric and the  curve metric \cite[Theorem~8.5]{Thurston1986} is a crowning achievement of Thurston's work on compact and cusped surfaces, but fails to be true in the setting of surfaces with boundary. Specifically, the distance between two points in the Teichm\"uller space between surfaces with boundary with respect to the Lipschitz metric is always positive, but may take negative values for the curve metric \cite{2009h} (when boundary length is variable). One method \cite{2009f} of remedying the potential negativity of the curve metric for surfaces with boundary, is to add the set of arcs joining boundary components to the set of curves (see Definition~\ref{d:arc}). Another approach is to constrain boundary lengths of the surfaces considered \cite[Theorem 7.9]{HS}.

In the present paper, we introduce new maps between hyperbolic surfaces which we call \emph{partial stretch maps}, which generalize Thurston's stretch maps to the setting of surfaces with boundary and allow us to study the Lipschitz metric geometry of the Teichm\"uller space of surfaces with boundary. 

\subsection{Notation}

 Before stating our main, results, we introduce some notation that will be used in the rest of this paper.

In this paper $S=S_{g,n}$ is an oriented topological surface of finite type and of negative Euler characteristic, with genus $g\geq0$ and $n\geq0$ (open) borders labeled from $1$ to $n$. We shall consider the following variants of the Teichm\"uller space of $S$:
\begin{itemize}
\item
$\mathcal{T}(S)$ is the space of homotopy classes of complete hyperbolic structures on $S$, where both cuspidal and hyperbolic boundary holonomy are admitted;
\item
$\mathcal{T}(S,\vec{b}=b_1,\ldots,b_n)$ is the subset of $\mathcal{T}(S)$ where for $k=1,\ldots,n$,  the geodesic representative of each $k$-th boundary has length $b_k\in[0,\infty)$ and where $b_k=0$ signifies a cusp;
\item
$\mathcal{T}:=\mathcal{T}(S,\vec{0})$ is the subset of $\mathcal{T}(S)$ consisting of homotopy classes of complete finite-area hyperbolic structures on $S$ (cusps are admitted, but not boundary components).
 \end{itemize}
 
We will, at times, need to consider the convex core of a complete hyperbolic surface $(S,h)$ of infinite area (with funnels) and we denote this convex core by $(\bar{S},\bar{h})$. In other words, $\bar{S}$  is a compact subsurface, with boundary, of $S$, obtained by cutting $(S,h)$ along the unique closed geodesic in the homotopy class of each boundary and removing all of the non-compact components (i.e.: the funnels), and $\bar{h}$ is obtained from $h$ by restricting to $\bar{S}$ a homotopy representative of $h$ which is geodesic on the boundary of $\bar{S}\subset S$. In this context, we adopt the notation $\mathcal{T}(\bar{S})$ and $\mathcal{T}(\bar{S},\vec{b})$ to refer to Teichm\"uller spaces of (finite-area) hyperbolic surfaces with geodesic boundary. Note that there is a natural identification between $\mathcal{T}(S)$ and $\mathcal{T}(\bar{S})$ and  between $\mathcal{T}(S,\vec{b})$ and $\mathcal{T}(\bar{S},\vec{b})$ via the map that takes a complete marked hyperbolic surface to its convex core.

\subsection{Paper outline and main results}

This paper is organized as follows: in \S\ref{s:review}, we briefly review Thurston's metric for surfaces without boundary, and the rest of the paper paper falls into two parts. The first part (\S\ref{s:torus} and \S\ref{s:one-holed}) is centered on the case of the one-holed torus, whereas the second part (\S \ref{applications}) applies to surfaces of greater topological generality. \medskip

In \S\ref{s:torus}, we construct partial stretch maps, which are natural generalizations of Thurston's stretch maps, for one-holed tori $S_{1,1}$. Specifically, we construct piecewise smooth homeomorphisms between complete hyperbolic metrics on $S_{1,1}$ with optimal Lipschitz constant.  Using partial stretch maps, we establish the following:

\medskip

\noindent\textbf{Theorem~\ref{thm:L=K}.} \textit{The   Lipschitz metric and the curve metric  on the Teichm\"uller space $\mathcal{T}(S_{1,1},b)$ coincide. Furthermore, this space, equipped with one of these metrics,  is a geodesic  space.}
 
\medskip

 \noindent\textbf{Theorem \ref{th:KA}.} \textit{The curve metric and the arc metric on $\mathcal{T}(S_{1,1},b)$ coincide if the boundary length $b$ satisfies $b\leq 4\operatorname{arcsinh}(1)$.}
 
 \medskip
 
In contrast, we have the following:

\medskip

 \noindent \textbf{Theorem \ref{th:inequality}.}  \textit{The arc metric on $\mathcal{T}(S_{1,1},b)$ is strictly greater than the curve metric for all sufficiently large $b$.}
 
\medskip
 
We use our construction of partial stretch maps on one-holed tori to obtain new classes of geodesics for the Thurston metric on Teichm\"uller spaces of surfaces without boundary and on Teichm\"uller  spaces of surfaces with boundary and with fixed boundary lengths. In particular, we give a construction of two-way Thurston geodesics (that is, segments which are geodesics when they are traversed in both senses) between triples of points in the Teichm\"uller space $\mathcal{T}(S_g)$ of a closed surface $S_g$ of any genus $\geq 2$ which produce arbitrarily ``fat'' metric triangles Teichm\"uller space. This leads us to address the question of the Gromov-hyperbolicity of the sum symmetrization of the Thurston metric, given by
\[
d_{\mathrm{sum}}(h_0,h_1):=K(h_0,h_1)+K(h_1,h_0), 
\] 
where $K$ denotes Thurston's curve metric on  $\mathcal{T}(S_g)$.\medskip

\noindent \textbf{Theorem~\ref{thm:nothyperbolic}.} \textit{The metric $d_{\mathrm{sum}}$  $\mathcal{T}(S_g)$ is not Gromov hyperbolic.}

\medskip

 We develop many of these ideas further in an upcoming paper \cite{HP} for surfaces of general (finite) topological type.

\subsection*{Acknowledgements}

We thank Kasra Rafi and David Dumas for correspondence and for his interest in our work, Fran\c{c}ois Gu\'{e}ritaud for his very thoughtful and enlightening responses to our many questions and Vincent Alberge who read a preliminary version of this paper. We also thank the referee whose corrections and suggestions substantially improved this paper.

\section{A review of Thurston's metric and the arc metric}\label{s:review}
 
 \subsection{Thurston's asymmetric metric}\label{s:asym}

Thurston defined in \cite{Thurston1986} two asymmetric metrics on the Teichm\"uller space $\mathcal{T}=\mathcal{T}(S,\vec{0})$ of $S$. We recall these definitions: 
 
\begin{definition}[Thurston's Lipschitz metric]\label{def:Lip}
Let $h_0$ and $h_1$ be two hyperbolic structures on $S$ and let $\varphi:(S,h_0)\to (S,h_1)$ be a homeomorphism homotopic to the identity map on $S$. The \emph{Lipschitz constant} $\hbox{Lip}(\varphi)$ of $\varphi$ is the quantity
\begin{displaymath}
\label{Lip}
\hbox{Lip}(\varphi)
:=\sup_{x\neq y\in S}
\frac{d_{h_1}\big{(}\varphi(x),\varphi(y)\big{)}}
{d_{h_0}\big{(}x,y\big{)}}.
\end{displaymath}
We denote by $L(h_0,h_1)$ the infimum of the Lipschitz constant over all homeomorphisms
  $\varphi: (S,h_0)\to (S,h_1)$ which are homotopic to the identity:
\begin{align}
\label{L}
L(h_0,h_1):=
\inf_{\varphi\sim\mathrm{id}_{S}}
\log\,\hbox{Lip}(\varphi).
\end{align}
The quantity $L(h_0,h_1)$ depends only on the homotopy classes of $h_0$ and $h_1$, and therefore descends to a function on $\mathcal{T} \times \mathcal{T}$. Thurston showed in \cite[\S 2]{Thurston1986} that $L$ satisfies all (distance) metric axioms except for symmetry. In fact, he showed the symmetry axiom fails: there exist hyperbolic structures $h_0$ and $h_1$ on $S$ such that $L(h_0,h_1)\neq L(h_1,h_0)$. For simplicity, we refer to such asymmetric metrics as metrics and we refer to $L(\cdot,\cdot)$ as the \emph{Lipschitz metric}.
\end{definition}

We denote by $\mathcal{S}$ the set of homotopy classes of essential simple closed curves on $S$, i.e. curves which are neither null-homotopic nor homotopic to a puncture. Note that with our definition, for   surfaces with boundary, an essential curve may be homotopic to a boundary component. Likewise, on a hyperbolic surface with funnels, an essential curve may be homotopic to the core curve of a funnel.

\begin{definition}[Thurston's curve metric] \label{def:curve}
Consider the quantity
\begin{align}
\label{K}
K(h_0,h_1):=
\sup_{\gamma\in\mathcal{S}}
\log\,\frac{l_{h_1}(\gamma)}{l_{h_0}(\gamma)}.
\end{align}
Thurston showed \cite[\S2]{Thurston1986} that $K$ also defines an asymmetric metric on $\mathcal{T}$. We refer to $K(\cdot,\cdot)$ as the \emph{curve metric}.
\end{definition}

For any hyperbolic structure $h$ on $S$, the length function $\l_h:\mathcal{S}\to\mathbb{R}$ defined on simple closed curves extends continuously to a function on the space $\mathcal{ML}$ of compactly supported measured laminations on $S$. (An early definition of this extension is contained in \cite[p. 24]{Kerckhoff}.) We denote this extension with the same notation $\l_h:\mathcal{ML}\to\mathbb{R}$. This function is positively homogeneous, that is, it satisfies $\l_h(r\mu)= r\l_h(\mu)$ for any $\mu$ in $\mathcal{ML}$  and for any $r>0$. By density of $\mathbb{R}_+\cdot\mathcal{S}$ in $\mathcal{ML}$, the supremum in (\ref{K}) can be taken over the elements of the space  $\mathcal{PML}$ of projective compactly supported laminations on $S$:
\begin{align}
\label{K1}
K(h_0,h_1):=
\sup_{[\mu]\in\mathcal{PML}}
\log\,\frac{l_{h_1}([\mu])}{l_{h_0}([\mu])}
\end{align}
where we denote by $[\mu]$ the projective equivalence class of an element $\mu$ in $\mathcal{ML}$, cf.   \cite[p. 4]{Thurston1986}. The main advantage of (\ref{K1}) over (\ref{K}) is that since $\mathcal{PML}$ is compact (this is a result of Thurston, cf. \cite{Thurston-BAMS} where the result is expressed in the equivalent form that uses measured foliations instead of measured laminations), the supremum in (\ref{K1}) is attained.

In the same paper, Thurston proved that $K\equiv L$  \cite[Theorem~8.5]{Thurston1986} and that this gives a geodesic metric on the Teichm\"uller space $\mathcal{T}$ of complete finite-area metrics on $S$.

 \subsection{$k$-Lipschitz maps and Thurston geodesics}
 
 Thurston constructed a class of distinguished geodesics for the metric $L$ (or, equivalently, for the metric $K$). His construction is based on certain Lipschitz maps between ideal hyperbolic triangles that we now recall. 
 
 Consider the most symmetric foliation by horocycles of the ideal triangle. This is a foliation of the three cusps of the triangle by horocyclic segments perpendicularly interpolating between boundary components, with a central unfoliated region bounded by three horocyclic segments which meet tangentially at their ends (see Figure~\ref{fig:stretch}). We refer to the three boundary points where two distinct horocycle leaves meet as \emph{anchor points}.
 
\begin{definition}[$k$-expansion map]
\label{def:stretch}
For any given $k\geq 1$, the $k$-expansion map between two ideal triangles is defined to be the identity on the central unfoliated region and  to send each horocycle at distance $d\geq 0$ from this central region onto the horocycle centered at the same point at infinity and at distance $kd$ from this unfoliated region, each horocycle being mapped linearly with respect to its parametrization by arclength. 
\end{definition}

\begin{figure}[!ht]
\centering
\includegraphics[width=0.8\linewidth]{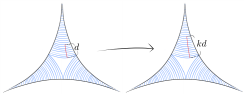}   
\caption{\small {The $k$-expansion map of an ideal hyperbolic triangle.}}   
\label{fig:stretch}
\end{figure}

Consider a geodesic lamination $\lambda$ which is maximal in the sense that it is not contained in any strictly larger geodesic lamination. In such a setting, maximality is equivalent to the fact that $S\setminus\lambda$ is a union of ideal triangles. Thurston utilized $k$-expansion maps to construct rays (that is, one-parameter families parametrized by $[0,\infty)$) of complete finite-area hyperbolic metrics on $S$
\[
h_t:=\operatorname{stretch}(h_0,\lambda,t)\text{, for }t\geq0,
\] 
where an initial hyperbolic metric $h_0$ is stretched along $\lambda$ by a factor of $e^t$, and replaces the metric on the complementary triangles with the pullback (hyperbolic) metric with respect to $e^t$-expansion maps on ideal triangles. Consequentially, the identity map 
\[
\mathrm{id}_S:(S,h_0)\to(S,h_t)
\]
is an $e^t$-Lipschitz map for every $t\geq0$. We refer to these maps as \emph{stretch maps}.\medskip

Thurston's stretch map construction is clearly well-defined when $\lambda$ consists of finitely many leaves, but careful analysis is required when $\lambda$ has (uncountably) infinitely many leaves --- Thurston did this in \cite[\S4]{Thurston1986}. Stretch maps varied over $t$ yield geodesic rays $(S,h_t)$ for the Thurston metric. We refer to geodesic segments of $(S,h_t)$ as \emph{stretch paths}; they are central to the results developed in \cite{Thurston1986}.

\subsection{Thurston metric for bordered surfaces}\label{s:thumetforbordered}

Neither the Lipschitz metric nor the curve metric, as respectively expressed by \eqref{L} and \eqref{K}, form (asymmetric) metrics on $\mathcal{T}(S)$ because they assume nonpositive values on particular ordered pairs of metrics, see \cite[Theorem~1.8]{GK}, \cite{Parlier} and  \cite[Theorem~2.4]{2009h}; in the last reference, it is shown that they assume negative values. However, \eqref{L} and \eqref{K} \emph{do} give positive (asymmetric) metrics when restricted to Teichm\"uller spaces $\mathcal{T}(S,\vec{b})$ of marked hyperbolic surfaces with fixed boundary holonomy \cite[Theorem 7.9]{HS}.

In \cite{GK}, Gu\'{e}ritaud and Kassel study an alternative form $L'$ of Thurston's Lipschitz metric whereby the infimum is taken over all Lipschitz maps homotopic to the identity, rather than just homeomorphisms. In the setting of working with surfaces without boundary, this alternative form of the Lipschitz metric is, in fact, equivalent. However, to the best of our knowledge, this is still open for bordered surfaces (see, for example, \cite[Conjecture~1.6 and \S2.2]{AD}). For this weakened Lipschitz metric, however, Gu\'{e}ritaud and Kassel \cite[Corollary~1.12]{GK} showed that $L'$ and $K$ are equal whenever either is positive. Combined with out previous remarks, we see that the metrics $K$ and $L'$ agree on $\mathcal{T}(S,\vec{b})$ and naturally generalize the Thurston metric. As a small technicality, Gu\'{e}ritaud and Kassel did not investigate the positivity of the curve ratio metric (or Lipschitz metric) on $\mathcal{T}(S,\vec{b})$, and instead achieve positivity by adjusting the Thurston metric with critical exponent renormalization factors \cite[Equation~(1.6)]{GK}. The benefit of their approach is that they obtained a metric for a very general class of representations (or rather, characters) which in turn encode very different geometric objects. Gu\'{e}ritaud has privately communicated to us a succinct alternative argument (to that used in \cite[Theorem 7.9]{HS}) for how to see the positivity of the na\"ive generalization of Thurston's metric on $\mathcal{T}(S,\vec{b})$ via his work with Kassel and Danciger \cite{DGK}.\medskip

In order to prove that $L'\equiv K$, Gu\'eritaud and Kassel established an equivariant form of the Kirszbraun--Valentine theorem as a machine for producing Lipschitz maps. It is unknown whether the maps so-produced are homeomorphic (or even injective), and it is tempting, therefore, to wonder if one might be able to reuse Thurston's stretch map construction to build optimal Lipschitz homeomorphisms and hence strengthen the result to $L\equiv K$. Unfortunately, stretch maps with respect to maximal geodesic laminations (i.e. those whose complementary regions are ideal triangles) generally distort boundary holonomy and hence do not lie in $\mathcal{T}(S,\vec{b})$. There are, however, special cases where one is able to recast Thurston's construction in a clever manner (e.g. Lenzhen, Rafi and Tao's paper \cite[\S6]{LRT}).

\subsection{The arc metric} In this subsection, $S$ is a surface with geodesic boundary equipped with a finite-area hyperbolic metric $h$. An \emph{arc} in $S$ is the homeomorphic image of a closed interval of $\mathbb{R}$ in $S$ such that the image of the interior (respectively the boundary)  of this interval is in the interior (respectively the boundary) of $S$. An arc in $S$  is said to be \emph{essential} if there is no disc in  this surface whose image is the union of this arc with a segment contained in the boundary $\partial S$. An \emph{orthogeodesic} in $S$ is a geodesic arc which makes a right angle at each  extremity, that is, at each intersection point with $\partial S$. 

Each boundary-relative homotopy class of arcs in $S$ contains a unique orthogeodesic; this is an analogue for the case of a surface with boundary of the result stating that each homotopy class of essential simple closed curve in $S$ contains a unique closed geodesic. (The former claim may be deduced from the latter by doubling the surface.)

We introduce the following notation:
\begin{itemize}
\item
$\mathcal{A}$ is the set of boundary-relative homotopy classes of arcs in $S$;
\item
$\mathcal{B}$ is the set of homotopy classes of  simple closed  curves in  $S$ that are homotopic to a boundary component.
\end{itemize}

The \emph{arc metric} on $\mathcal{T}(S)$, first defined in \cite{2009f}, based on lengths of simple orthogeodesics and boundary components, is an asymmetric metric on $\mathcal{T}(S)$ which is an analogue of the Thurston metric defined in \cite{Thurston1986} on the Teichm\"uller space of surfaces without boundary. We recall its definition:

\begin{definition}[The arc metric] \label{d:arc}
Given two hyperbolic metrics $h_0$ and $h_1$ on $S$, consider the quantity
\begin{align}
\label{A}
A(h_0,h_1):=\sup_{\alpha\in\mathcal{A}\cup \mathcal{B}}
\log\,\frac{l_{h_1}(\alpha)}{l_{h_0}(\alpha)},
\end{align}
where for each $\alpha$ in $\mathcal{A}$, $l_{h}(\alpha)$ denotes the length of the unique orthogeodesic representative of $\alpha$ in its boundary-relative homotopy class. 

From the definition, we see that the quantity $A(h_0,h_1)$ does not depend on the homotopy classes of the metrics $h_0$ and $h_1$, therefore the same formula induces a function $A:\mathcal{T}(S)\times \mathcal{T}(S)\to\mathbb{R}$. Like the Thurston metric for surfaces without boundary,  $A$ is an asymmetric metric on the Teichm\"uller space of $S$ (see \cite{2009f}), and we refer to it as the \emph{arc metric} on this space.
\end{definition}

It follows from the definitions that the set $\mathcal{B}$  is contained in the set $\mathcal{S}$ of homotopy classes of simple closed curves in $S$.  Later in this paper, we shall use equation (\ref{A}) to define a metric on the Teichm\"uller space of a one-holed torus on which the hyperbolic structures have fixed boundary length. In this case, the supremum on the right hand side in (\ref{A}) may be taken on $\mathcal{A}$ instead of $\mathcal{A}\cup \mathcal{B}$.

The definition of the arc metric makes it formally an analogue to the Thurston metric for surface with boundary. But there are deeper relations between the arc metric and the Thurston metric. For example, the arc metric  is equal to the pullback of the Thurston metric with respect to the doubling embedding $\mathcal{T}(S)\hookrightarrow\mathcal{T}(S^d)$, where $S^d$ is the surface without boundary obtained by doubling $S$ along its boundary \cite[Corollary~2.9]{2009f}. Furthermore, there is a precise sense in which the Thurston metric on the Teichm\"uller space of a surface with cusps is a limit of arc metrics on Teichm\"uller spaces of surfaces with boundary, as the lengths of the boundary components tend to zero, see \cite{PS}.

The following theorem gives an alternative (and in some sense, the original \cite[Definition~2.1]{2009f}) definition of the arc metric:

\begin{theorem}
\label{AKcompare}
For any two elements  $h_0$ and $h_1$ in the Teichm\"uller space $\mathcal{T}(S)$, we have
\[A(h_0,h_1)=
\sup_{\alpha\in\mathcal{A}\cup\mathcal{S}}
\log\,\frac{l_{h_1}(\alpha)}{l_{h_0}(\alpha)}.\]
 
\end{theorem}
This theorem is proved in \cite[Proposition~2.13]{2009f}. The equality  $A\geq K$ follows immediately from the definition. The reverse inequality  is based on the fact that any sequence of arcs in $S$ which are Dehn-twisted a high number of times about a simple closed curve $\gamma$ roughly detects the length ratio for $\gamma$.

The study of geodesics for the arc metric has hitherto been based on explicit constructions of Lipschitz maps between right angled hexagons (see \cite{2009j} and its generalization in \cite{2015-Yamada1}). Recently, other constructions were obtained by Alessandrini and Disarlo, see \cite{AD}.

\section{Partial stretch maps on ideal Saccheri quadrilaterals and one-holed tori}
\label{s:torus}

\subsection{Saccheri quadrilaterals}
Classically, a \emph{Saccheri quadrilateral} in the hyperbolic plane is a geodesic convex quadrilateral with two opposite sides of equal length perpendicular to a common third side. Saccheri quadrilateral admit two real dimensions worth of moduli, and hence play a hyperbolic geometric role somewhat akin to that of rectangles in Euclidean geometry.  See the quadrilateral $ABCD$ in the left-hand side of Figure~\ref{fig:saccheri}, where the two equal sides $AD$ and $BC$ are perpendicular to $AB$. The angles at $C$ and $D$ are then necessarily equal and acute. On the right-hand side of this figure, we have represented an \emph{ideal Saccheri quadrilateral}, by which we mean that the vertices $C$ and $D$ are ideal points (that is, points at infinity of the hyperbolic plane) and (hence) the sides $AD$ and $BC$ have infinite length. In this case, the angles at $C$ and $D$ are $0$.\medskip

The isometry type of the ideal Saccheri quadrilateral $ABCD$ is determined by the length of $AB$. Note that an ideal triangle may be regarded as a limit of a family of such ideal Saccheri quadrilaterals in which the length of the side $AB$ tends to $0$.

\begin{figure}[!ht]
\centering
\includegraphics[width=0.6\linewidth]{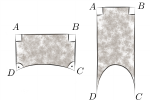}
\caption{\small {A Saccheri quadrilateral (left) and an ideal Saccheri quadrilateral (right)}}    
\label{fig:saccheri}
\end{figure}

 We shall also work with \emph{extended ideal Saccheri quadrilaterals}. These are the infinite-area quadrilaterals obtained by extending an ideal Saccheri quadrilateral $ABCD$ of Figure \ref{fig:saccheri}
 (regarded as being embedded in the hyperbolic plane $\mathbb{H}^2$) to the complete infinite-area convex domain bordered by $CD$ and the two bi-infinite geodesics respectively containing  $AD$ and $BC$.

\subsection{Partial horocyclic foliations}
An extended ideal Saccheri quadrilateral has two ideal vertices. We foliate the neighborhood of each such vertex with horocycles centered at that   vertex. We extend this foliation in a reflection-symmetrical manner until the two foliations meet tangentially at a point of the bi-infinite edge (see Figure \ref{fig:horocyclic}). We refer to this intersection point as an \emph{anchor point}. This reflection-symmetric partial foliation of the quadrilateral is uniquely determined, and in the special case where the length of $AB$ is $0$, it corresponds to ``two out of three sectors'' of the horocyclic foliation of the ideal triangle employed by Thurston (Figure \ref{fig:stretch}).

Any ideal Saccheri quadrilateral is equipped with a horocyclic foliation induced by the one of the extended ideal Saccheri quadrilateral which contain it, see Figure~\ref{fig:horocyclic}. Cases  (a), (b) and (c) in this figure illustrate each of the three (mutually exclusive) situations that may arise with regards to the partial horocyclic foliation:
\begin{itemize}
\item
(a) occurs when the length of $AB$ is smaller than $2\operatorname{arcsinh}(1)$;
\item
(b) occurs when this length is equal to $2\operatorname{arcsinh}(1)$, and
\item
(c) occurs when this length is greater than $2\operatorname{arcsinh}(1)$.
\end{itemize}

\begin{figure}[!ht]
\centering
\includegraphics[width=0.80\linewidth]{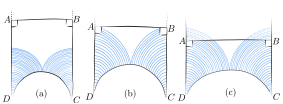}    
\caption{\small {Partial horocyclic foliations of various ideal Saccheri quadrilaterals and their extensions (dotted lines).}}    
\label{fig:horocyclic}
\end{figure}
%
%
%
\subsection{$k$-expansion maps}

 For any $k\geq 1$, we define the  \emph{$k$-expansion maps} of any extended ideal Saccheri quadrilateral  in much the same way as Thurston did for ideal triangles (see Definition~\ref{def:stretch}). Specifically, for a given $k\geq 1$, this map is equal to the identity map on the unfoliated region and sends any leaf of the horocyclic foliation situated at distance $d\geq 0$ from the unfoliated region to the one at distance $kd$ from that region, mapping linearly with respect to arclength on each horocyclic leaf (see Figure~\ref{fig:Saccheri-stretch}). 
 
%
%
%
%
%
 
\begin{figure}[!ht]
\centering
\includegraphics[width=0.8\linewidth]{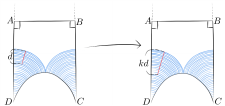}    
\caption{\small {The $k$-expansion map of an ideal Saccheri quadrilateral}}    
\label{fig:Saccheri-stretch}
\end{figure}

\begin{proposition} Consider an ideal Saccheri quadrilateral and the extended quadrilateral which contains it.  We have the following:

(1) The $k$-expansion map of the extended ideal Saccheri quadrilateral has Lipschitz constant precisely equal to $k$. 
 
 (2) In the cases (a) and (b) of Figure~\ref{fig:horocyclic}, this $k$-expansion map of the extended ideal Saccheri quadrilateral  induces a homeomorphism of the original ideal Saccheri quadrilateral. This induced homeomorphism has Lipschitz constant $k$.
\end{proposition}

\begin{proof} 
The same techniques that are used to show that ideal triangle stretch maps are $k$-Lipschitz apply here as well. They are based on the existence of orthogonal preserved partial foliations, apply in this case as well. Proof details are omitted by Thurston in \cite[Proposition~2.2]{Thurston1986} but are easily recovered from (for example) the computation made in \cite[\S 2]{2009j}. We may, in essence, ignore the unfoliated region. On the horocyclically foliated regions, the $k$-expansion map expands the orthogonal geodesic foliation by $k$ and contracts the horocyclic foliation. The orthogonality of the two invariant foliations ensures $k$-Lipschitz-ness.

\end{proof}

\subsection{Partial stretch maps on one-holed tori}\label{sub:partial-stretch}
In the  rest of this section, we set $S=S_{1,1}$ in notation such as $\mathcal{T}(S,b)$ to denote the Teichm\"uller space of one-holed hyperbolic tori with prescribed boundary holonomy.\medskip 

We now construct $k$-Lipschitz maps between hyperbolic one-holed tori using $k$-expansion maps between ideal Saccheri quadrilaterals.

In Lemma \ref{thm:chainrecurrent}, we shall make use of Thurston's theory of train tracks and train track approximation of measured geodesic laminations as it is presented in \S 8.9 of his Princeton lecture notes \cite{T-Notes3}. A measured geodesic lamination is said to be \emph{rational} if its support is the union of simple closed geodesics. Otherwise it is said to be \emph{irrational}. In the train track coordinates charts, this distinction corresponds to the usual distinction between rational and irrational coordinates of the lamination, up to a multiplicative constant.

A compactly supported geodesic lamination on a hyperbolic surface is said to be \emph{chain recurrent}  if it is the limit of simple closed geodesics in the Hausdorff topology of the set of compact subsets of the surface. The support of a chain recurrent geodesic lamination, as a Hausdorff limit of compact subsets, is necessarily compact. We refer the reader to \cite{Thurston1986} p. 25 for the basic properties of chain recurrent laminations.

\begin{lemma}[Chain recurrence characterization]
\label{thm:chainrecurrent}
A one-holed torus admits in the interior of its convex core three types of geodesic laminations  which are chain recurrent; see Figure~\ref{fig:chainrecurrence}:
\begin{enumerate}
\item
a simple closed geodesic $\gamma$;
\item
the union of a simple closed geodesic $\gamma$ and a bi-infinite geodesic $\ell$ which spirals to $\gamma$ from one side of $\gamma$ and to $\gamma^{-1}$ from the other side;
\item
a geodesic lamination $\mu$ with uncountably many leaves corresponding to the support of some irrational measured geodesic lamination $[\mu]$. 
\end{enumerate}
\end{lemma}

\begin{proof}
We first note that all three of these types of geodesic laminations are in fact chain recurrent: case~{(1)} is definitional, case~{(2)} follows from the fact that the spiraling bi-infinite leaf can be written as the Hausdorff limit of a sequence obtained from performing arbitrarily many Dehn twists to any simple closed geodesic transverse to $\gamma$, and case~{(3)} follows from the well-known theorem asserting that homotopy classes of weighted simple closed geodesics  are dense in the space of measured geodesic laminations for the measure topology of this space. In particular, $[\mu]$ is the limit of some sequence of weighted curves with support $\{\gamma_n\}$. Quoting Thurston's proof of \cite[Proposition 8.10.3]{T-Notes3}, ``[for] any point $x$ in the support of a measure $[\mu]$ and any neighborhood $U$ of $x$, the support of a measure close enough to $[\mu]$ must intersect $U$'', we see that $\mu$ must be a subset of the Hausdorff limit $\nu$ of (a subsequence of) the $\gamma_n$. The collar lemma tells us that $\nu$ is supported on the interior of the convex core of $S$. However, $\mu$ is a maximal lamination on the interior of the aforementioned convex core, and hence $\mu$ is precisely equal to $\nu$ and is therefore chain recurrent.\medskip


Conversely, Let $\lambda$ be any chain recurrent geodesic lamination. We know that every compactly supported lamination contains a sublamination which supports a transverse measure (Proposition 8.10.6 of \cite{T-Notes3}). Any such sublamination of $\lambda$ falls either into cases~$(1)$ or $(3)$, which respectively correspond to the support of rational and irrational measured laminations on $S$. Taking a train track approximating this lamination,  one can see that the complement of an irrational measured lamination on a one-holed torus is an annulus homotopy equivalent to the boundary of the convex core of $S$. This implies that case~$(3)$ is already maximal among chain recurrent laminations. On the other hand, simple closed geodesics can be extended while preserving the property of being chain recurrent in precisely two ways, both of which fall into case~$(2)$. This covers all possibilities for $\lambda$.
\end{proof}

\begin{figure}[!ht]
\centering
\includegraphics[width=0.8\linewidth]{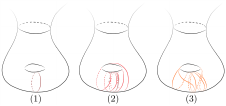}    
\caption{\small {The three types of chain recurrent laminations classified in Lemma~\ref{thm:chainrecurrent}}, the rightmost figure depicts a train-track carrying an irrational lamination.}
\label{fig:chainrecurrence}
\end{figure}

\begin{theorem}[Partial stretch maps]
\label{thm:pstretch}
For any complete hyperbolic metric $h_0$ on $S$ representing an element of $\mathcal{T}(S,b)$ and for any chain recurrent lamination $\lambda$ on $S$, there is a ray
\begin{align}
h_t:=\operatorname{pstretch}(h_0,\lambda,t)\text{, for }t\geq0,
\end{align}
of complete hyperbolic metrics in $\mathcal{T}(S,b)$ such that  
\begin{itemize}
\item
for $0\leq t$, the identity map $\mathrm{id}_S:(S,h_0)\rightarrow(S,h_t)$ is an $e^{t}$-Lipschitz homeomorphism which is an isometry outside of a compact set contained in the convex core of $S$; \item
the geodesic representative of $\lambda$ for $h_0$ is also geodesic for $h_t$, and the identity map 
preserves this set and expands arclength along $\lambda$ by a factor of $e^{t}$.
\end{itemize}
In addition, the path of equivalence classes in $\mathcal{T}(S,b)$  of metrics $(S,h_t)$, for $t\geq 0$, is a geodesic ray for both the Lipschitz metric $L$ and the curve metric $K$, with $L\equiv K$ along $(S,h_t)$. We refer to (equivalence classes of) subsegments of $(S,h_t)$ as \emph{partial stretch paths}.
\end{theorem}

\begin{remark} \label{rem:invariant}
Note that in this paper we do not define partial stretch maps for arbitrary maximal (compactly supported) laminations. In particular, we require the stretched lamination to be invariant under the hyperelliptic involution $\iota$ on $S$. Fortunately, this property is satisfied by all chain recurrent laminations as they are Hausdorff limits of simple closed geodesics, and the latter are all invariant under $\iota$.
\end{remark}

In the proof below, we shall talk about \emph{crowned hyperbolic surfaces}. Such a surface is, by definition, homeomorphic to a compact annulus with a certain number of deleted points on one of its boundary components, equipped with a hyperbolic structure so that the boundary component with no points deleted is a closed geodesic, and the other boundary component consists of a union of bi-infinite geodesics converging from each side to a deleted point. See Figure \ref{fig:crown}.

\begin{proof}[Proof of Theorem \ref{thm:pstretch}]
Lemma~\ref{thm:chainrecurrent} lets us deal with this construction on a case-by-case basis.\medskip

\noindent\textbf{When $\lambda$ is irrational.} This corresponds to case~$(3)$ of Lemma~\ref{thm:chainrecurrent}. The irrational lamination $\lambda$ is fixed under the hyperelliptic involution $\iota$ on $S$ since it is the limit of simple closed geodesics --- which are fixed under $\iota$ (Remark~\ref{rem:invariant}). There are precisely two (simple) orthogeodesic rays $\sigma_1,\sigma_2$ on the convex core of $S$ which launch from the boundary of the convex core of $S$ and spiral toward $\lambda$, and the set $\sigma_1\cup\sigma_2$ is therefore also fixed by $\iota$. Since $\iota$ acts non-trivially on the boundary of the convex core of $S$, the end-points of $\sigma_1$ and $\sigma_2$ must be permuted under $\iota$ and hence $\iota$ permutes $\sigma_1$ and $\sigma_2$. Cutting the convex core of $S$ along $\lambda$, $\sigma_1$ and $\sigma_2$ produces two ideal Saccheri quadrilaterals $Q_1,Q_2$. In particular, due to the $\iota$-invariance of $\sigma_1\cup\sigma_2$, the involution $\iota$ defines an isometry between these two ideal Saccheri quadrilaterals. This in turn means that the extended ideal Saccheri quadrilaterals $\widehat{Q}_i$, obtained by cutting $S-\lambda$ along the bi-infinite geodesics extending $\sigma_1$ and $\sigma_2$, are isometric via $\iota$. We consider the partial horocyclic foliation on  the crowned hyperbolic surface $S-\lambda$ gotten by  gluing the horocyclic partial foliations of the two completed Saccheri ideal quadrilaterals (see Figure~\ref{fig:horocyclic}).  The $k$-expansion map of these quadrilaterals  that use these partial foliations as in Figure \ref{fig:Saccheri-stretch} can be glued together to give a $k$-expansion map on $S-\lambda$ with Lipschitz constant $k$.\medskip

\begin{figure}[!ht]
\centering
\includegraphics[width=0.8\linewidth]{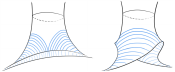}    
\caption{\small {The partial horocylic foliation on $S-\lambda$, viewed from two different perspectives.}}    
\label{fig:crown}
\end{figure}

We appeal now to Thurston's construction in \cite[Proposition~4.1]{Thurston1986}. To begin with, given any hyperbolic structure in the neighborhood of a geodesic lamination $\mu$ on a general hyperbolic surface, we obtain a measured foliation transverse to $\mu$ as well as additional data in a horocyclic neighborhood of each cusp of the surface. Thurston presents this data in the form of a function that he dubbed a \emph{sharpness function}. His  \cite[Proposition~4.1]{Thurston1986} then reverses this procedure to recover a hyperbolic structure. We now make use of this reverse procedure, but first note that there are no cusps on $S$, and hence we need not touch upon the notion of sharpness functions.



We now produce the family $h_t$ of complete hyperbolic metrics on $S$. First, we observe that we can employ \cite[Proposition~4.1]{Thurston1986} essentially without alteration by doubling the convex core of $S$, noting that any sufficiently small neighborhood $N_\epsilon(\lambda)$ of $\lambda$ is also doubled as $\lambda$ is compact and supported on the interior of the convex core. We use the previously constructed $k$-expansion map on $S-\lambda$, with $k=e^t$, to redefine the metric outside of $\lambda$. 
Doing this for each $e^t$ produces  a family $h_t$ of hyperbolic metrics on $S$. Since the $e^t$-expansion map is equal to the identity outside of a compact set, the metrics $h_t$ are all isometric outside of a compact set. The developing map then tells us that the boundary holonomy is independent of $t$. Note  however that if the foliated region on which the partial stretch of the infinite-area surface is made (represented in Figure \ref{fig:crown}) trespasses beyond the geodesic representative of the boundary of the convex core of that surface, then the restriction to the (initial) convex core of the partial stretch path of metrics produces a path of hyperbolic metrics whose boundary components are no longer geodesic. (In particular, the part of the boundary of the initial  convex core which cuts across the foliated region is metrically deformed so that it is no longer geodesic.) Thus, this construction does not give a path in the Teichm\"uller space of the compact surface with boundary and with constant boundary length.

\medskip

\noindent\textbf{When $\lambda$ is a simple closed geodesic.} This corresponds to case~$(1)$ of Lemma~\ref{thm:chainrecurrent}. In this scenario, we extend $\lambda$ to a chain recurrent lamination comprised of $\lambda$ and a bi-infinite simple geodesic spiraling to it from one side and $\lambda^{-1}$ from the other. Thus, we have reduced this case to:
\medskip

\noindent\textbf{Remaining case.} This corresponds to case~$(2)$ of Lemma~\ref{thm:chainrecurrent}, where $\lambda$ is the union $\gamma\cup\ell$ of a simple closed geodesic $\gamma$ and a bi-infinite geodesic $\ell$ spiraling to $\gamma$ on one side of $\gamma$ and to $\gamma^{-1}$ on the other side, for some chosen orientation on $\gamma$. Equivalently, considering that the spiraling of an infinite ray toward $\gamma$ induces an orientation on this closed curve, we are assuming that the two disjoint rays spiraling toward $\gamma$ from each side induce on it different orientations.
We first observe that both $\gamma$ and $\ell$ are preserved under the hyperelliptic involution $\iota$ on $S$. And just as with the case when $\lambda$ is irrational, the crowned hyperbolic surface $C:=S-(\gamma\cup\ell)$ is obtained by gluing together two isometric extended ideal Saccheri quadrilaterals (see again Figure \ref{fig:crown}) whose partial horocyclic foliations perfectly align, since they both meet $\lambda$ perpendicularly. Again, the upshot is that we obtain a $k$-expansion map on $C$ which expands along its boundary components by a factor of $k$. As with the last step of the proof for the irrational $\lambda$ case, we define $h_t$ on $C$ by pullback with respect to the $e^t$-expansion map, and extend $h_t$ over $\lambda=\gamma\cup\ell$ either by invoking \cite[Proposition~4.1]{Thurston1986} or more na\"ively by observing that the pullback metric on $C$ glues continuously on the tangent spaces over $\gamma$ and $\ell$. Note that this na\"ive argument fails to be rigorous if $\lambda$ is irrational as $\lambda$ contains more than just the boundary geodesics on $S-\lambda$.\footnote{There are, in fact, uncountably many leaves in $\lambda$, and only finitely many lie on the boundary.}

\medskip

\noindent\textbf{Geodesic ray $(S,h_t)$.} To complete the proof, we show that the family of equivalence classes of metrics $(S,h_t)$ for $t\geq 0$ is a geodesic ray in $\mathcal{T}(S,b)$ for both $L$ and $K$. By construction, the stretch map $\mathrm{id}_S: (S,h_0)\to (S,h_t)$, for $t\geq0$, is an $e^{t}$-Lipschitz map which stretches $\lambda$ by a factor of $e^{t}$. This means that the measured lamination support in $\lambda$ realizes the maximum curve ratio between $(S,h_0)$ and $(S,h_t)$ for all time and $K(h_0,h_t)=t$. Moreover, observe that the composition of the $e^s$-expansion map and the $e^t$-expansion map on the ideal Saccheri quadrilateral is precisely the $e^{s+t}$-expansion map
. This tells us that the stretch map $\mathrm{id}_S:(S,h_0)\to(S,h_s)$ composed with the stretch map $\mathrm{id}_S:(S,h_s)\to(S,h_{s+t})$ is precisely equal to the stretch map $\mathrm{id}_S:(S,h_0)\to(S,h_{s+t})$. Therefore, we see that the equivalence class of 
\begin{align}
K(h_s,h_{s+t})=t\text{ for all }s,t\geq0,
\end{align}
and hence the equivalence classes of $(S,h_t)$, for $t\geq 0$, form a geodesic ray for the curve metric $K$. Finally, the Lipschitz metric is at least the curve metric. Thus, we have
\begin{align}
t\geq L(h_s,h_{s+t})\geq K(h_s,h_{s+t})=t,
\end{align}
where the first inequality follows from the stretch map $\mathrm{id}_S:(S,h_s)\to(S,h_{s+t})$ being $e^t$-Lipschitz. This in turn tells us that the equivalence class of $(S,h_t)$ is a geodesic ray for $L$.
\end{proof}

The partial measured foliations on the two ideal Saccheri quadrilaterals, when they are glued together,  define a measured foliation class $F$ on the torus with one boundary component.
When the Teichm\"uller space of this surface with boundary is equipped with its Thurston boundary, we have the following:

\begin{theorem}
The geodesic ray $(S,h_t)$ defined in Theorem \ref{thm:pstretch} converges, as $t\to\infty$,  to the projective class $[F]$ of  $F$, considered  as an element of Thurston's boundary of the Teichm\"uller space of the torus with one boundary component. 
\end{theorem}
\begin{proof}[Sketch of proof] The same argument used in the proof of \cite[Theorem 3.7]{Papa1988} which established the analogous property for a surface without boundary applies here. It is ultimately hinged on a double inequality \cite[Proposition 3.1]{Papa1988} comparing hyperbolic length and intersection number for all curves in $(S,h_t)$.
\end{proof}

\section{The Lipschitz, arc and curve metrics on the Teichm\"uller space of the one-holed torus}\label{s:one-holed}

\subsection{The Lipschitz metric versus the curve metric on $\mathcal{T}(S=S_{1,1},b)$}
\label{sec:K=L}

In this subsection we establish a generalization of Thurston's \cite[Corollary~8.5]{Thurston1986}. 

\begin{theorem}
\label{thm:L=K}
 The   Lipschitz metric and the curve metric  on the Teichm\"uller space $\mathcal{T}(S_{1,1},b)$, defined respectively by \eqref{L} and \eqref{K},  coincide. Furthermore, this space, equipped with such a metric,  is a geodesic  space.

\end{theorem}

As previously discussed in section~\ref{s:thumetforbordered}, \cite[Corollary~1.12]{GK} show that $L'\equiv K$ for a version $L'$ of the Lipschitz metric that infimizes Lipschitz constants over a greater class of maps than homeomorphisms. This, combined with Thurston's result that $L\equiv K$ for finite-area complete hyperbolic surfaces (without borders) raises the hypothesis that perhaps $L\equiv K$ still holds true over the Teichm\"uller space of bordered hyperbolic surfaces. We resolve this in the affirmative for $\mathcal{T}(S_{1,1},b)$ by adapting Thurston's proof in the case of surfaces without boundary, which relies fundamentally upon \cite[Theorem~8.2]{Thurston1986} and \cite[Theorem~8.4]{Thurston1986}. We start with the following definition, which is a correction of the definition taken from \cite[paragraph before Theorem~8.2]{Thurston1986}:

\begin{definition}[Ratio-maximizing laminations]
Given a pair of marked hyperbolic metrics $h_0,h_1$ in $\mathcal{T}(S,b)$, we say that a (non-necessarily measured) geodesic lamination $\lambda$ is \emph{ratio-maximizing} if there exists a Lipschitz homeomorphism with optimal (i.e. minimal) Lipschitz constant $k=e^{K(h_0,h_1)}$, mapping from a neighborhood of $\lambda$ in $(S,h_0)$ to a neighborhood of $\lambda$ in $(S,h_1)$, which is in the correct homotopy class. Moreover, we require that such an optimal Lipschitz homeomorphism necessarily takes the (geodesic realisation of) leaves of $\lambda$ in $(S,h_0)$ to the corresponding (geodesic realisation of) leaves of $\lambda$ in $(S,h_1)$, locally stretching tangent vectors along each leaf of $\lambda$ by a factor of $k$.
\end{definition}

We already know from \cite[Proposition~4.1 or Theorem~8.1]{Thurston1986} that such a neighborhood map exists if $\lambda$ is contained in the support of a projective measured lamination maximizing $K$.\footnote{This measured lamination is unique when $S=S_{1,1}$.} Thurston's insightful observation is that his notion of ``(length-)ratio-maximizing'' laminations can be extended to chain recurrent (non-necessarily measured) geodesic laminations containing isolated bi-infinite geodesic leaves provided that one requires also that a neighborhood of the lamination be mapped across so as to preserve the optimal Lipschitz constant. Fortunately, the proofs for \cite[Theorems~8.2 and 8.4]{Thurston1986} hold in our context, and thus we conclude:

\begin{theorem}[{\cite[Theorems~8.2 and 8.4]{Thurston1986}}]
\label{thurstonthm}
There is a unique chain recurrent geodesic lamination which is ratio-maximizing and contains all other ratio-maximizing chain recurrent geodesic laminations for the pair $h_0,h_1$. We refer to it as the \emph{maximal ratio-maximizing lamination} and denote it by $\mu(h_0,h_1)$. These laminations have the following property: if $\{h_0^{(i)}\}$ and $\{h_1^{(i)}\}$ are sequences of complete hyperbolic structures in $\mathcal{T}(S,b)$ which respectively converge to $h_0$ and $h_1$, then $\mu(h_0,h_1)$ contains any lamination in the limit set of $\mu(h_0^{(i)},h_1^{(i)})$ with respect to the Hausdorff topology on the set of geodesic laminations on $(S,h)$.
\end{theorem}

\begin{remark}
Although the above result is stated for sequences of metrics $\{h_0^{(i)}\}$ and $\{h_1^{(i)}\}$, the result obviously also applies to continuous paths 
\[
\hat{h}_0(t),\hat{h}_1(t):[0,1]\to\mathcal{T}(S,b)
\] 
of metrics which evaluate to $\hat{h}_0(1)=h_0$ and $\hat{h}_1(1)=h_1$ at time $t=1$. 
\end{remark}

\begin{lemma}\label{firsttime}
Given $h_0,h_1\in\mathcal{T}(S,b)$, let $\lambda$ be any maximal geodesic lamination that contains the maximal ratio-maximizing lamination $\mu(h_0,h_1)$ as a sublamination. Consider the partial stretch path
\[
\hat{h}_t:=\operatorname{pstretch}(h_0,\mu(h_0,h_1),tK(h_0,h_1)).
\]
Then either $\mu(\hat{h}_{t},h_1)$ is constant, or there is a first time $t_1\in(0,1)$ such that $\mu(\hat{h}_{t_1},h_1)$ differs from $\mu(h_0,h_1)$.
\end{lemma}

\begin{proof}
Let us assume, for a proof by contradiction, that $\mu(\hat{h}_{t},h_1)$ is not constant and that there is no such first time. Then, let $t_1$ be the infimum of all times $t\in[0,1)$ such that $\mu(\hat{h}_t,h_1)$ differs from $\mu(h_0,h_1)$. Either $t_1=0$, or there is a small interval $(0,t_1)$ such that for all $s\in(0,t_1)$, $\mu(\hat{h}_{s},h_1)=\mu(h_0,h_1)$. Theorem~\ref{thurstonthm} ensures that in both  cases, $\mu(h_0,h_1)$ is a sublamination of $\mu(\hat{h}_{t_1},h_1)$. By assumption, there is a sequence of times $s$ decreasing to $t_1$ for which $\mu(\hat{h}_s,h_1)$ differs from $\mu(h_0,h_1)$. For any such $s$ sufficiently close to $t_1$, Theorem~\ref{thurstonthm} tells us that $\mu(\hat{h}_s,h_1)$ lies in a small neighborhood $N$ of $\mu(\hat{h}_{t_1},h_1)$.\medskip

The construction of the partial stretch map asserts that there is an $e^{(1-s)K(h_0,h_1)}$-Lipschitz map from $N$ on $\hat{h}_s$ to $N$ on $h_1$ which realises the length ratio between $\hat{h}_s$ and $h_1$ of some measured lamination supported on $\mu(h_0,h_1)$. In particular, this means that all other measured laminations supported on $N$ have lower length ratio (between $\hat{h}_s$ and $h_1$) than $e^{(1-s)K(h_0,h_1)}$. Now, since $\mu(\hat{h}_s,h_1)$ lies in $N$ and contains the support of some measured lamination, the optimal Lipschitz constant $e^{K(\hat{h}_s,h_1)}$ associated to $\mu(\hat{h}_s,h_1)$ being a ratio-maximizing lamination means that $e^{K(\hat{h}_s,h_1)}\leq e^{(1-s)K(h_0,h_1)}$. On the other hand, since $N$ contains only a subset of the measured laminations on $S$, $e^{(1-s)K(h_0,h_1)}$ must be less than or equal to $e^{K(\hat{h}_s,h_1)}$. Therefore, $e^{K(\hat{h}_s,h_1)}= e^{(1-s)K(h_0,h_1)}$. We therefore see that mapping $\mu(\hat{h}_s,h_1)$ stretches it by $e^{(1-s)K(h_0,h_1)}$ and hence $\mu(\hat{h}_s,h_1)$ is a sublamination of $\mu(h_0,h_1)$. Conversely, the fact that $e^{K(\hat{h}_s,h_1)}= e^{(1-s)K(h_0,h_1)}$ means that $\mu(h_0,h_1)$ is a ratio-maximizing lamination on $\hat{h}_s$, and is hence definitionally a sublamination of $\mu(\hat{h}_s,h_1)$. Therefore, $\mu(h_0,h_1)=\mu(\hat{h}_s,h_1)$ for all $s>t_1$ sufficiently close to $t_1$, which is a contradiction.
\end{proof}

We are now well-equipped to prove that the Lipschitz metric and the curve metric agree (Theorem~\ref{thm:L=K}).

\begin{proof}[Proof of Theorem~\ref{thm:L=K}]
We first note that $\mu(h_0,h_1)$ cannot contain the boundary of the convex core as every geodesic arc on $\mu(h_0,h_1)$ is definitionally stretched by a factor of $k=e^{K(h_0,h_1)}$, whereas the total length of the boundary remains fixed. Indeed, since the boundary length is always fixed, every ratio-maximizing (chain recurrent) lamination we encounter during the course of this proof must lie within the interior of the convex core. Since $\mu(h_0,h_1)$ is chain recurrent and supported on the interior of the convex core, it lies in one of the three types of laminations described in Lemma~\ref{thm:chainrecurrent}.\medskip

\noindent\textbf{Type~(2) and (3) laminations:} when $\mu(h_0,h_1)$ is a maximal chain recurrent lamination (i.e. case~$(2)$ and $(3)$ in Lemma~\ref{thm:chainrecurrent}), there is a unique partial stretch path
\[
\hat{h}_t:=\operatorname{pstretch}(h_0,\mu(h_0,h_1),tK(h_0,h_1))
\]
which stretches along $\mu(h_0,h_1)$. We claim that this geodesic ray must reach $h_1$ at time $1$. Assume not, then by Lemma~\ref{firsttime}, the maximal ratio-maximizing lamination $\mu(\hat{h}_t,h_1)$ must differ from $\mu(h_1,h_0)$ for some first time $t_1\in(0,1)$. However, Theorem~\ref{thurstonthm} then ensures that $\mu(h_0,h_1)$ is a proper subset of $\mu(\hat{h}_{t_1},h_1)$, which is impossible due to the maximality of $\mu(h_0,h_1)$ among all chain recurrent laminations (supported on the interior of the convex core).\medskip

\noindent\textbf{Type~(1) laminations:} the remaining case is when $\mu(h_0,h_1)$ is a simple closed geodesic $\gamma$ (i.e. type~(1) in Lemma~\ref{thm:chainrecurrent}), in which event there are precisely two maximal chain recurrent laminations $\mu_\pm$ containing $\mu(h_0,h_1)$ which are supported on the convex core interior. We partially stretch along $\mu_+$ without loss of generality, and note that we cannot reach $h_1$ or else $\mu_+$ would be a maximal ratio-maximizing lamination and hence we would have $\gamma=\mu(h_0,h_1)=\mu_+$. Thus, by Lemma~\ref{firsttime} and Theorem~\ref{thurstonthm}, there is a first time $t_1\in(0,1)$ where the maximal ratio-maximizing lamination between $\mathrm{pstretch}(h_0,\mu_+,t_1K(h_0,h_1))$ and $h_1$ becomes a proper chain recurrent superlamination of $\gamma$. The only possible candidates are either $\mu_+$ or $\mu_-$, which are both maximal among chain recurrent laminations. This reduces our proof to the previous case, and we see that we will eventually reach $h_1$ partially stretching with respect to this new lamination. In particular, since leaves of all measured lamination supported inside of $\mu(h_0,h_1)$ are uniformly and maximally stretched through the entire concatenated stretch path, this process must end precisely at time $t=1$. As a minor aside, note that the second concatenated partial stretch path cannot be stretched with respect to $\mu_+$ as then we would have $\gamma=\mu(h_0,h_1)=\mu_+$, and hence must be with respect to $\mu_-$. \medskip

We have constructed geodesics for both $L$ and $K$ which join arbitrary points $(S,h_0)$ and $(S,h_1)$ in $\mathcal{T}(S,b)$. Moreover, by construction, the Lipschitz constant of the partial stretch map at $t=1$ is $e^{K(h_0,h_1)}$. Therefore, the two metrics must be equal, as desired.
\end{proof}

The following corollary is immediate from the proof of Theorem~\ref{thm:L=K}.

\begin{corollary}
Any two points in $\mathcal{T}(S,b)$ are joined by a Thurston geodesic which is the concatenation of at most two partial stretch paths.
\end{corollary}

\subsection{The curve metric versus the arc metric on $\mathcal{T}(S,b)$}
\label{sec:KvsA}

\begin{theorem}\label{th:KA}
If $b\leq 4\operatorname{arcsinh}(1)$, then the curve metric $K$ and the arc metric $A$ on $\mathcal{T}(S,b)$ are equal.
\end{theorem}

\begin{proof}
When $b\leq 4\operatorname{arcsinh}(1)$, the partial horocyclic foliation (as depicted by the blue lines in Figures~\ref{fig:horocyclic}, \ref{fig:Saccheri-stretch}, and \ref{fig:crown}) is completely contained in the convex core of $S$ and hence any partial stretch map simply evaluates to being the identity map on the convex core boundary. In particular, this says that the Lipschitz constant for the optimal Lipschitz map $\mathrm{id}_S: (S,h_0)\to(S,h_1)$  constructed in \S \ref{sub:partial-stretch} is the same as the constant for the optimal Lipschitz map $\mathrm{id}_{\bar{S}}: (\bar{S},\bar{h}_0)\to(\bar{S},\bar{h}_1)$ which extends it, and hence we have the following chain of inequalities:
\begin{align*}
L(h_0,h_1)=K(h_0,h_1)\leq A(\bar{h}_0,\bar{h}_1)\leq L(\bar{h}_0,\bar{h}_1)=L(h_0,h_1),\notag
\end{align*}
where the first inequality is explained in Remark~\ref{AKcompare} and the second inequality is a general consequence of Lipschitz metrics being at least as great as length-ratio-type metrics.
\end{proof}

In contrast:

\begin{theorem}\label{th:inequality}
The arc metric is strictly greater than the curve (or the arc)  metric on $\mathcal{T}(S,b)$ for all sufficiently large $b$. That is to say, there exist hyperbolic structures $h_0,h_1\in\mathcal{T}(S,b)$ such that
\[
K(h_0,h_1)< A(h_0,h_1).
\]
\end{theorem}

\begin{proof}
We sketch the construction required for this proof. Consider a right-angled hexagon $H_0$ with alternating sidelengths
\[
2\operatorname{arcosh}(x^4),\operatorname{arcsinh}(x^2)\text{ and }\operatorname{arcsinh}(x^2);
\]
and consider another right-angled hexagon $H_1$ with alternating sidelengths
\[
2\operatorname{arcosh}(x^4),\operatorname{arcsinh}(x^3)\text{ and }\operatorname{arcsinh}(x^3),
\]
where $x$ is understood to be a very large number (see Figure~\ref{fig:hexagons}). \medskip

\begin{figure}[!ht]
\centering
\includegraphics[width=0.75\linewidth]{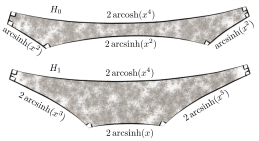}   
\caption{\small {Two ``skinny'' hexagons $H_0$ and $H_1$.}}   
\label{fig:hexagons}
\end{figure}

We double each $H_i$ to obtain a pair of pants $P_i$ in such a way that the above listed sides remain unglued and form the boundary of the $P_i$. Then, in each $P_i$, we glue the two boundary components of equal length with no twisting to form two geodesic-bordered one-holed tori $\bar{T}_0$ and $\bar{T}_1$ in $\mathcal{T}(\bar{S},4\operatorname{arcosh}(x^4))$. The two shortest interior simple closed geodesics $\alpha,\beta$ (see Figure~\ref{fig:thintorus}) on $\bar{T}_0$ are both of length $2\operatorname{arcsinh}(x^2)$, whereas on $\bar{T}_1$ they are respectively of lengths $2\operatorname{arcsinh}(x)$ and $2\operatorname{arcsinh}(x^3)$.\medskip

\begin{figure}[!ht]
\centering
\includegraphics[width=0.5\linewidth]{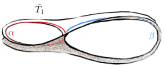}   
\caption{\small {The ``skinny'' torus $\bar{T}_1$ and its two shortest geodesics of lengths $l_\alpha=2\operatorname{arcsinh}(x)$ and $l_\beta=2\operatorname{arcsinh}(x^3)$.}}   
\label{fig:thintorus}
\end{figure}

We regard $\alpha$ and $\beta$ as standard $\mathbb{Z}$-basis vectors $(1,0)$ and $(0,1)$ for $H_1(\bar{T}_i;\mathbb{Z})=\mathbb{Z}^2$. The primitive elements of $H_1(\bar{T}_i;\mathbb{Z})$ naturally biject with the collection of simple closed geodesics on $\bar{T}_i$. In particular, since these tori are so thin, the hyperbolic length of a geodesic $\gamma$ whose homology class is $(p,q)$ is approximately
\begin{align*}
2\operatorname{arcsinh}(x^2)|p|+2\operatorname{arcsinh}(x^2)|q|&\text{ on }\bar{T}_0\text{, and}\\
2\operatorname{arcsinh}(x)|p|+2\operatorname{arcsinh}(x^3)|q|&\text{ on }\bar{T}_1.
\end{align*}
Let $T_i$ denote the infinite-area complete hyperbolic extension of $\bar{T}_i$. For large $x$, the inverse hyperbolic sine function is close to the logarithm function, and one sees that $K(T_0,T_1)\approx\frac{3}{2}$ and in particular is bounded independently of $x$. On the other hand, $A(\bar{T}_0,\bar{T}_1)$ grows without bound as $x\to\infty$ because the unique (orthogeodesic) arc on $\bar{S}$ that does not intersect $\beta$ has length, with respect to $\bar{T}_0$ and $\bar{T}_1$, given by
\[
2\operatorname{arcsinh}
\left(
\sqrt{\tfrac{x^4+1}{x^8-1}}
\right)
=
2\operatorname{arcsinh}
\left(
\tfrac{1}{\sqrt{x^4-1}}
\right)
\approx
\tfrac{2}{x^2}
\text{ and }
2\operatorname{arcsinh}
\left(
\sqrt{\tfrac{x^6+1}{x^8-1}}
\right)
\approx
\tfrac{2}{x}.
\]
These expressions are easily derived from hyperbolic trigonometric formulae (e.g. \cite[Theorem~2.3.4]{Buser}).

\end{proof}

\section{Applications to general Teichm\"uller spaces}\label{applications}

\subsection{Novel geodesics on $\mathcal{T}(S_{g,n},\vec{b})$}
\label{sec:novel}

In general, describing and establishing that a map is optimal Lipschitz even for simple examples is difficult. The partial stretch maps we describe in \S\ref{s:torus} resolve this in a concrete way for \emph{complete} hyperbolic one-holed tori with fixed boundary holonomy. Moreover, the fact that they are isometric outside of a compact set easily enables additional gluing-map-based constructions, allowing for a rich family of novel Thurston geodesics. We give the following examples:

\begin{enumerate}
\item
When the boundary geodesic representative satisfies $b\leq4\operatorname{arcsinh}(1)$, the \emph{stretch region} (i.e. the set of points on the surface where there is \emph{any} metric distortion) of the partial stretch map between the two hyperbolic structures lies within the convex core of the surface. We may double this convex core along its boundary component to yield partial stretch maps on closed genus $2$ surfaces. Varying the expansion factor $k=e^t$, we get a new class of Thurston geodesics in the Teichm\"uller space $\mathcal{T}(S_2)$ equipped with Thurston's metric.
To see that this differs from a Thurston stretch map-induced geodesic, we first observe that the length of the central curve stays constant along the above partial stretch map-induced geodesic. Furthermore, every point along our geodesic is isometric with respect, to reflection in the central curve, which means that the maximal ratio-maximizing lamination is also reflection symmetric. In particular, the first stretch path segment for a Thurston stretch map-induced geodesic is stretched along a maximal ratio-maximizing lamination which is also maximal as a geodesic lamination aside from the possible addition of leaves on punctured monogons (which cannot occur in this setting, as the complement of the maximal ratio-maximizing lamination on the two halves of the surface is either a $4$-holed sphere or a  annulus with punctures on its boundary components), and the reflection symmetry forces it to contain the central geodesic. This means that such a stretch map would expand the length of the central curve, thereby distinguishing it from our partial stretch map-induced geodesic.

\item
Whenever a topologically stable neighborhood of the stretch region of the partial stretch map may be isometrically embedded in another (not necessarily hyperbolic) surface $\Sigma$ of greater topological complexity, we obtain partial stretch maps on $\Sigma$ by setting the new Lipschitz map to be the identity on the complement of the embedded stretch region. One simple, but potentially useful instance of this arises (again) when $b\leq4\operatorname{arcsinh}(1)$ and we extend $S$ to $\Sigma=S_{g,n}$ by gluing on a genus $g-1$ surface with ${n+1}$ holes of lengths $(b,\vec{b}')$. This construction produces new geodesics for the Thurston metric in $\mathcal{T}(S_{g,n},\vec{b}')$.
\item
We may adapt Example~$(2)$ to work in greater generality by performing small metric deformations with Lipschitz constant smaller than or equal to the Lipschitz constant on $S$ (such as small Fenchel--Nielsen twists) in such a way as to not disturb the metric expansion along the embedded stretch region of $S\subset\Sigma$ . One example of this comes from adapting Example~$(1)$ when $4\operatorname{arcsinh}(1)<b\leq4\operatorname{arcosh}(\frac{3}{2})$, where we double the surface (including its extruding stretch region) and reglue to a genus $2$ surface with a $\frac{b}{4}$ twist. This enables the extruding stretch loci to fit on top of (originally) unstretched domains. The $b=4\operatorname{arcosh}(\frac{3}{2})$ instance of this particular construction is used by Lenzhen, Rafi and Tao in their proof of \cite[Theorem~1.1]{LRT} (the length of $b=4\operatorname{arcosh}(\frac{3}{2})$ can be computed using Figure~10 of their paper). Another example of this is simultaneously performing stretching along $\lambda$ as well as small earthquakes along measured laminations which do not transversely intersect $\lambda$. In any case, this is not a new idea, but a potentially useful one.
\item
In the present paper, all of our stretch maps on one-holed tori $S$ are invariant under the hyperelliptic involution $\iota$ and hence descend to $S/\iota$---a hyperbolic sphere with three $\pi$-cone angles and one hole with geodesic representative of length $\frac{b}{2}$ (see Figure~\ref{fig:foliations}$(2)$). This cone surface is also the quotient, with respect to a $\mathbb{Z}_2\times\mathbb{Z}_2$ symmetry, of a four-holed sphere $S_{0,4}$ with boundary geodesic representatives of length $\frac{b}{2}$ (see Figure~\ref{fig:foliations}$(3)$), and hence all relevant stretch maps lift to $S_{0,4}$ with four equal length boundary components. This new $S_{0,4}$ building block immediately affords new flexibility in gluing-map based constructions. 
\item
We further observe that in Example~$(4)$ above, in instances where the stretching lamination contains a simple closed curve $\gamma$ (see Figure~\ref{fig:foliations}), this $\gamma$ then descends to an interval $\gamma'$ of length $\frac{\ell_\gamma}{2}$ joining cone points in $S/\iota$ and lifts to a separating geodesic $\gamma''$ in $S_{0,4}$ of length $2l_\gamma$. Since $\gamma''$ lies on the stretching lamination, it remains geodesic under stretching, and so we may cut $S_{0,4}$ along $\gamma''$ to obtain two isometric pairs of pants $P$ of boundary lengths $2l_{\gamma},\frac{b}{2},\frac{b}{2}$. In particular, the partial stretching map on $P$ increases the length of just one of its boundary components. This gives an $S_{0,3}$ building block for potential gluing-map based Lipschitz map constructions.

\item
Finally, continuing from the previous example, we may glue the unstretched boundary components of $P$ to then obtain a stretching map on a one-holed torus which increases the boundary length but preserves the length of (at least) one interior simple closed geodesic. For $b\leq4\operatorname{arcsinh}(1)$, some of these examples may be used to construct boundary-length fixing paths in Teichm\"uller space, thereby yielding completely novel arc metric geodesics (thanks to Theorem~\ref{th:KA}).
\end{enumerate}

\begin{figure}[ht!]
\centering
\includegraphics[width=1\linewidth]{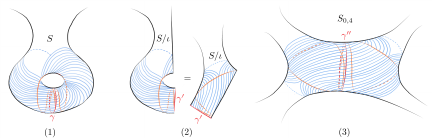}   
\caption{\small {Partial horocycle foliations (blue) on $(1)$ the one-holed torus of boundary length $L$; $(2)$ the quotient one-holed sphere with three $\pi$-cone points with boundary length $\frac{L}{2}$; $(3)$ the four-holed sphere with all boundary components of length $\frac{L}{2}$. The maximally stretched lamination consists of two geodesics: the support of the transverse measure (red) and the geodesic which spirals towards the former (orange).}} 
\label{fig:foliations}
\end{figure}

\begin{remark}
\label{rmk:curious}
The construction in Example~$(4)$ suffices to show that
\[
\left(\mathcal{T}(S_{1,1},2b),L\equiv K\right)
\text{ is isometric to }
\left(\mathcal{T}(S_{0,4},b,b,b,b),L\equiv K\right).
\]
The specialized setting when $b=0$ positively answers the first of the three cases in a question posed by Walsh in \cite[Paragraph after Theorem~7.9]{W}. We doubt that the other two cases posed, $\mathcal{T}(S_{1,2},\vec{0})$ versus $\mathcal{T}(S_{0,5},\vec{0})$ and $\mathcal{T}(S_2)$ versus $\mathcal{T}(S_{0,6},\vec{0})$, are isometric, however, it is plausible that these pairs may be made isometric, provided that one permits cone-point boundaries for $\mathcal{T}(S_{0,5},\vec{b})$ and $\mathcal{T}(S_{0,5},\vec{b})$.
\end{remark}
  
\subsection{The metric geometry of the Thurston metric.}
\label{sec:nonhyp}
We further illustrate the versatility of partial stretch maps by proving the following result inspired by \cite[Theorem~1.1]{LRT}. In this statement, a \emph{two-way geodesic} is a geodesic which is also a geodesic when traversed in the reverse direction. Note that every geodesic for a symmetric metric is two-way.

\begin{theorem}[Arbitrarily non-thin geodesic triangles]
\label{thm:nonthin}
For every $g\geq 2$ and for every $D>0$, there are points $W,X,Y\in \mathcal{T}(S_g)$ joined by two-way Thurston geodesic segments $G_{WX}$, $G_{WY}$ and $G_{XY}$ such that there is a point $Z\in G_{XY}$ which is distance at least $D$ away from $G:=G_{WX}\cup G_{WY}$ (i.e. the distances $K(Z,G)$ and $K(G,Z)$ are both greater than $D$). Moreover, $G$ is also a two-way Thurston geodesic between $X$ and $Y$.
\end{theorem}

\begin{proof}
We first give the proof for $g=2$. Fix $D>0$ and choose any $b<4\mathrm{arcsinh}(1)$. The main construction is the following: take a constant-speed geodesic arc $h_t:[0,1]\to \mathcal{T}(S=S_{1,1}, b)$, and glue $h_t$ to an orientation reversed family of metrics $\bar{h}_{1-t}$ along the boundary curve, with some amount of twisting $\phi_t$. We claim that if $b$ is small, if the arc is long, and if $\phi_t$ does not vary too much in $t$, this path is a two-way geodesic. Making two choices for the function $\phi_t$ that are, roughly speaking, as different as possible while meeting these constraints gives two geodesics with the same endpoints but which are very far apart at the midpoint, as required.\medskip

We consider the double $(S^d,h^d)$ of the convex hull of $(S,h)$ and denote by $\gamma$ the separating simple closed geodesic we glued along to get $S^d$. There is a small collar neighborhood $C$ around $\gamma$ that is outside of the stretch region and hence unaffected by partial stretch maps on either of the components of $S^d-\gamma$. Construct a smooth family of Lipschitz maps $\phi_t:C\rightarrow C_t$ where $C_t$ is the Fenchel--Nielsen twist of $C$ by $t$. For $\tau\gg2D\operatorname{arcsinh}(1)/b$, any metric on $S^d$ that comes from gluing (possibly with twisting) the convex cores of two $1$-holed tori in $\mathcal{T}(S,b)$, when Fenchel--Nielsen twisted by $\phi_{\tau}$ on $C$, will give a map with Lipschitz constant $\operatorname{Lip}(\phi_\tau)\gg D$. We reparametrize and rescale such a family of Fenchel--Nielsen twist deformations $\phi_{t\in[0,\tau]}$ on $C$ so that the Lipschitz constant for $\phi_t$ increases linearly with respect to $t$ and finishes at $\tau=1$.\medskip

Choose arbitrary points $h_0,h_1\in\mathcal{T}(S,b)$ such that, 
\begin{itemize}
\item
$K(h_0,h_1)\geq K(h_1,h_0)$, and the curve length ratio $K(h_0,h_1)$ is realised by a very short simple closed curve $\alpha$ on $S$. In particular, the length of $\alpha$ on both $(S,h_0)$ and $(S,h_1)$ is close to $0$,
\item
$K(h_0,h_1)> 2\operatorname{Lip}(\phi_\tau)\gg 2D$.
\end{itemize}

Consider a Thurston geodesic $h_t:[0,1]\to\mathcal{T}(S,b)$ joining $h_0$ and $h_1$ obtained via a concatenation of partial stretch maps (hence no metric distortion ever occurs in a small collar neighborhood around the boundary geodesic). Furthermore, consider $\bar{h}_t$, the same path of metrics in $\mathcal{T}(S,b)$ but with the orientations on the surface reversed. Construct a path of metrics in $\mathcal{T}(S_2)$ by gluing $h_t$ to $\bar{h}_{1-t}$, and set the endpoints of this path as $X$ and $Y$ (see Figure~\ref{fig:geodesics}). Since we can independently construct optimal Lipschitz maps on the left and the right halves of the surface with Lipschitz constants depending solely on $K(h_0,h_1)$ and $K(h_1,h_0)$, we need only consider simple closed curves which lie on these two half surfaces rather than on the entire surface when determining the optimal Lipschitz constant for the entire surface.\medskip

The path so produced is necessarily a geodesic in both directions. Going from $X$ to $Y$, the Lipschitz constant on the left of $\gamma$ is governed by the growth in the length of $\alpha$ (going from $\ell_\alpha(h_0)$ to $\ell_\alpha(h_1)$) versus the Lipschitz constant on the right of $\gamma$ being governed (roughly) by the logarithm of the reciprocal of $\alpha$ (going from $-\log\ell_\alpha(h_1))$ to $-\log\ell_\alpha(h_0)$). Thus, the path is dominated by the left of $\gamma$ when traversing from $X$ to $Y$ and the right of $\gamma$ when going from $Y$ to $X$. Set the midpoint of this geodesic as $W$ and define $G_{WX}$ as the two-way Thurston geodesic between $X$ and $W$ and $G_{WY}$ as the two-way Thurston geodesic between $W$ and $Y$.\medskip

\begin{figure}[!ht]
\centering
\includegraphics[width=1\linewidth]{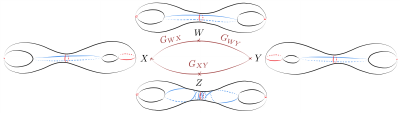}   
\caption{\small {A depiction of the closed genus $2$ hyperbolic surfaces $W,X,Y,Z$ on the geodesics $G_{WX},G_{WY}$ and $G_{XY}$. Going from $X$ to $Y$, the leftmost (red) curve is maximally stretched; from $Y$ to $X$ the rightmost (red) curve is maximally stretched; $Z$ is distance $D$ away from/to $G=G_{WX}\cup G_{WY}$ because of the twisted (blue) curve.}}   
\label{fig:geodesics}
\end{figure}

Next, we produce another two-way Thurston geodesic between $X$ and $Y$ by gluing $h_t$ and $\bar{h}_{1-t}$ with a ``twist'': we apply $\phi_t$ on $C$ until half-way (i.e. $t=\tfrac{1}{2}$) and then unwind the Fenchel--Nielsen twist with $\phi_{-t}$ until $t=1$. Since $K(h_0,h_1)\gg 2 D$ and both geodesics are uniformly parametrized (and we are producing actual Lipschitz maps for $t\in[0,1]$), the Fenchel-Nielsen twist is small enough so that the optimal Lipschitz constant from $X$ to $Y$ is unaffected along this new path, thus ensuring that it is a two-way Thurston geodesic. We label this path as $G_{XY}$ and denote the $t=\tfrac{1}{2}$ midpoint by $Z$. The Fenchel-Nielsen twist introduced along $G$ means that $Z$ is necessarily at least distance $D$ from every single point in $G$ with distances measured in either direction: consider the length of the blue curve depicted in Figure~\ref{fig:geodesics} for the distances from $G_{WX}\cup G_{WY}$ to $Z$, likewise consider a (relatively) short curve on $Z$ which minimally intersects $\gamma$ but is disjoint from the left and right $\alpha$s to bound (from below) the distances from $Z$ to $G_{WX}\cup G_{WY}$.\medskip

For general $g$, the proof is essentially the same: one simply needs to glue an additional ``unstretched'' $S_{g-1,2}$ between the convex hulls of $(S,h_0),(S,h_1)$ which cap off the two ends.
\end{proof}

Roughly speaking, the above result says that there is no na\"ive sense in which the Thurston metric can be Gromov hyperbolic  \cite{Gromov, CDP}. However, without wanting to clarify what $\delta$-hyperbolicity might mean for an asymmetric metric, we instead make the following concrete statement 
in which we use the classical notion of Gromov hyperbolicity defined by the Gromov product for  spaces with a symmetric distance function, and make no assumptions about the metric being geodesic:

\begin{corollary}
\label{thm:nothyperbolic}
The sum-symmetrization $d_{\mathrm{sum}}(h_0,h_1):=K(h_0,h_1)+K(h_1,h_0)$ of the Thurston metric on $\mathcal{T}(S_g)$ is not Gromov hyperbolic.
\end{corollary}

\begin{proof}[Proof of Corollary~\ref{thm:nothyperbolic}]
Two-way geodesics for the Thurston metric are geodesics for $d_{\mathrm{sum}}$. Hence, the edges $G_{WX},G_{WY},G_{XY}$ constitute the edges of a geodesic triangle $\triangle_{WXY}$ for $d_{\mathrm{sum}}$. It is evident that $d_{\mathrm{sum}}\geq K$, and hence $\triangle_{WXY}$ is not $D$-thin. By choosing $D>0$ to be arbitrarily large, this contradicts the condition that all geodesic triangles (if any exist) be $\delta$-thin for some $\delta>0$, which in turn is a necessary (but not necessarily sufficient, due to the potential sparsity of geodesic triangles in metric spaces which might not be geodesic) condition for Gromov's $\delta$-hyperbolicity. 
\end{proof}

\begin{remark}
 Brock and Farb, in their paper  \cite{BF}, give conditions on a metric on Teichm\"uller space which is invariant by the mapping class group action to be non-Gromov hyperbolic. Their result holds for  geodesically complete path metrics on Teichm\"uller, and cannot be used in our context since we do not know whether the sum-symmetrization of Thurston's asymmetric metric belongs to this class of metrics.
 \end{remark}
 
 We also highlight the following result which follows from the statement of Theorem~\ref{thm:nonthin}: The sum-symmetrized metric of Thurston's metric is not uniquely geodesic. 

\begin{remark}
Our proofs for Theorem~\ref{thm:nonthin} and Corollary~\ref{thm:nothyperbolic} are fairly flexible, and one extends easily to $\mathcal{T}(S_{g,n},\vec{b})$ by gluing in combinations of unstretched surfaces or stretched $S_{0,4}$ (Example~$(4)$ of \S\ref{sec:novel}) as needed. In particular, we can use the same construction in the following cases:
\begin{itemize}
\item
$g\geq2$ with arbitrary $n\geq0$;
\item
$g=1$ and $n\geq2$, with the condition that at least $4-n$ of the boundary components must have geodesic representatives of the same length $b<2\operatorname{arcsinh}(1)$;
\item
$g=0$ and $n\geq4$, with the condition that at least $8-n$ of the boundary components must have geodesic representatives of the same length $b<2\operatorname{arcsinh}(1)$,
\end{itemize}
to obtain generalizations of Theorem~\ref{thm:nonthin} and Corollary~\ref{thm:nothyperbolic}. We leave the construction of these cases to interested readers.
\end{remark}

This (informal) non-hyperbolicity of the Thurston metric (Theorem~\ref{thm:nonthin}) essentially comes from the fact that the \emph{envelope} from $X$ to $Y$ in $\mathcal{T}(S_2)$, i.e. the union of all the geodesics from $X$ to $Y$  (a notion studied in \cite{DLRT}), becomes very ``fat'' as the distance between $X$ and $Y$ increases. Indeed, the geodesics $G$ and $G_{XY}$ constructed in the proof of Theorem~\ref{thm:nonthin} fail to fellow-travel. Moreover, due to the generality of our construction, either:
\begin{enumerate}
\item
there are no arbitrarily long arc metric geodesic segments in the thick part of Teichm\"uller space, or
\item
it is impossible to have a quasi-``thin triangles''-type claim of the same form as \cite[Theorem~E]{R}, where if a side of a geodesic triangle is in the thick part of Teichm\"uller space, then it lies near to one of the other two sides. 
\end{enumerate}
Nonetheless, it seems plausible that one might recover a weaker notion of hyperbolicity such as: there exists some $\delta>0$ so that any three points $X,Y,Z$ are the vertices of some $\delta$-thin triangle. Or, perhaps the stronger statement that the envelope from $X$ to $Y$ lies within the union of the $\delta$-neighborhoods of the envelopes from $X$ to $Z$ and from $Z$ to $Y$.


\begin{thebibliography}{99}
 
\bibitem{AD} D. Alessandrini and V. Disarlo, \textit{Generalized stretch lines for surfaces with boundary}, arXiv:1911.10431.



\bibitem{BF} J. Brock and B. Farb, \textit{Curvature and rank of Teichm\"uller space},
Amer. Jour. of Mathematics
 128, (2006) N. 1, 1-22.
 
 
 \bibitem{Buser} P. Buser, \textit{Geometry and Spectra of Compact Riemann Surfaces}, Progress in Mathematics, vol. 106, Birkh\"{a}user Boston Inc., Boston, MA, 1992. 

 
\bibitem{CDP} M. Coornaert, T. Delzant,  A. Papadopoulos, \emph{G\'eom\'etrie et th\'eorie des groupes}. Les groupes hyperboliques de Gromov, Lecture Notes in Mathematics, 1441, Springer-Verlag,  1990. 


\bibitem{DGK} J. Danciger, F. Gu\'{e}ritaud, and F. Kassel, \textit{Margulis spacetimes via the arc complex}, Invent. Math. \textbf{204} (2016), no.~1, 133--193.
 
\bibitem{DLRT} D. Dumas, A. Lenzhen, K. Rafi, and J. Tao, \textit{Coarse and fine geometry of the Thurston metric},  Forum Math. Sigma 8 (2020), Paper No. e28, 58 pp. 

\bibitem{Gromov} M. Gromov,  \textit{Hyperbolic groups}. In Essays in group theory, 75-263, Math. Sci. Res. Inst. Publ., 8, Springer, New York, 1987.
 
\bibitem{GK} F. Gu\'{e}ritaud and F. Kassel, \textit{Maximally stretched laminations on geometrically finite hyperbolic manifolds}, Geom. Topol. \textbf{21} (2017), no.~2, 693--840.

\bibitem{HP} Y. Huang, and A. Papadopoulos, \textit{Optimal Lipschitz maps on bordered hyperbolic surfaces and the Thurston metric theory of Teichm\"{u}ller space}, in preparation.

\bibitem{HS} Y. Huang and Z. Sun, \textit{McShane identities for higher Teichm\"uller theory and the Goncharov-Shen potential}, Memoirs of the American Mathematical Society, to appear.


\bibitem{Kerckhoff} S. P. Kerckhoff, Earthquakes are analytics, Comm. Math. Helv. 60 (1985), p. 17-10.

\bibitem{LRT} A. Lenzhen, K. Rafi and J. Tao, \textit{The shadow of a Thurston geodesic to the curve graph}, J. Topol. \textbf{8} (2015), 1085--1118.

\bibitem{Papa1988} A. Papadopoulos, \textit{Sur le bord de Thurston de l'espace de Teichm\"uller d'une surface non compacte}, Math. Ann.  \textbf{282} (1988), 353--359.

\bibitem{2009f}  L. Liu,  A. Papadopoulos, W. Su and G. Th\'eret, \textit{On length spectrum metrics and weak metrics on Teichm\"uller spaces of surfaces with boundary}, Ann. Acad. Sci. Fenn. Math. \textbf{35} (2010), no.~1, 255--274.


\bibitem{PS} A. Papadopoulos and W. Su, \textit{Thurston's metric on Teichm\"uller space and the translation distances of mapping classes}, Ann. Acad. Sci. Fenn. Math. \textbf{41} (2016), no.~2, 867--879. 
  
\bibitem{2009h} A. Papadopoulos and  G. Th\'eret, \textit{Shortening all the simple closed geodesics on surfaces with boundary}, Proc. Amer. Math. Soc. \textbf{138} (2010), 1775--1784.  

\bibitem{2009j}  A. Papadopoulos and  G. Th\'eret, \textit{Some Lipschitz maps between hyperbolic surfaces with applications to Teichm\"uller theory}, Geom. Dedicata \textbf{150} (2011), no.~1, 233--247. 

\bibitem{2015-Yamada1} A. Papadopoulos and S. Yamada, \textit{Deforming hexagons and the arc and the Thurston metric on Teichm\"uller space}, Monatsh. Math. \textbf{172} (2017), no.~1, 97--120.

\bibitem{Parlier} H. Parlier, Lengths of geodesics on Riemann surfaces with boundary. Ann. Acad. Sci. Fenn. Math. 30 (2005), no. 2, 227--236.

\bibitem{R} K. Rafi, \textit{Hyperbolicity in Teichm\"{u}ller space}, Geom. Topol. \textbf{18} (2014), no.~5, 3025--3053.



 \bibitem{T-Notes3} W. P. Thurston,  The geometry and topology of three-manifolds, Lecture notes, Princeton, 1979. Available at http://www.msri.org/publications/books/gt3m/
 
 

\bibitem{Thurston1986} W. Thurston, \textit{Minimal stretch maps between hyperbolic surfaces}, arXiv:9801039 [math.GT], 1986.

\bibitem{Thurston-BAMS} W. P. Thurston, On the geometry and dynamics of diffeomorphisms of surfaces, Bulletin Amer. Math. Society, 19 (1988), p. 417-431.


\bibitem{W} C. Walsh, \textit{The horoboundary and isometry group of Thurston's Lipschitz metric}, Handbook of Teichm\"uller theory. Vol. IV, IRMA Lect. Math. Theor. Phys., 19,  Eur. Math. Soc., Zurich, 2014.
 


\end{thebibliography}
\end{document}